\documentclass{article}
\usepackage[utf8]{inputenc}
\usepackage{verbatim}
\usepackage[english]{babel}
\usepackage{graphicx}
\usepackage{xcolor}
\usepackage[pdfborder={0 0 0}]{hyperref}
\usepackage{lipsum}
\usepackage{enumerate}
\usepackage{import}
\usepackage{url}
\usepackage{physics}
\usepackage{geometry}
\usepackage{calc}
\usepackage{amsfonts}
\usepackage{mathtools}
\usepackage{amssymb}
\usepackage{amsmath}
\usepackage{amsthm}
\usepackage{todonotes}
\usepackage{hyperref}
\newcommand{\N}{\mathbb{N}}

\newcommand{\C}{\mathbb{C}}

\newcommand{\diam}{\text{diam}}
\DeclareMathOperator{\prop}{prop}
\DeclareMathOperator{\id}{id}
\newcommand{\Cenv}{C^*_{\text{env}}}
\newtheorem{theorem}{Theorem}[section]
\newtheorem{proposition}[theorem]{Proposition}
\newtheorem{lemma}[theorem]{Lemma}
\newtheorem{corollary}[theorem]{Corollary}
\theoremstyle{definition}
\newtheorem{definition}[theorem]{Definition}
\theoremstyle{remark}
\newtheorem{remark}[theorem]{Remark}

\usepackage{tikz}
\usetikzlibrary{matrix}
\usepackage{tikz-cd}
\usetikzlibrary{cd}

\usepackage{subcaption}

\title{Operator systems for tolerance relations on finite sets}
\author{Mick Gielen and Walter D. van Suijlekom}
\date{}

\begin{document}
\maketitle

\begin{abstract}
    We study the duals of a certain class of finite-dimensional operator systems, namely the class of operator systems associated to tolerance relations on finite sets or equivalently the class of operator systems that are associated with graphs. In the case where the graphs associated with these operator systems are chordal we are able to find concrete realizations of their duals as sitting inside of finite-dimensional $C^*$-algebras. We then use these concrete realizations to compute the $C^*$-envelopes, propagation numbers and extremal rays of these duals in the chordal case. Finally, we exemplify our results by applying them to operator systems of band matrices.
\end{abstract}

\section{Introduction}
Completely positive maps have long since played an important rule in the study of operator algebras. The appropriate context for the study of such maps is in the setting of operator systems, which were first introduced in \cite{injectivity}. Since their inception, operator systems have found applications outside of operator algebra theory. They appear under the name of ``non-commutative graphs" in the context of quantum information theory in \cite{quantum}. Recent developments in non-commutative geometry also crucially make use of operator systems \cite{spectralTruncations, toleranceRelations}.

Two equivalent notions of operator systems exist. A concrete operator system is a unital self-adjoint subspace of the space $B(H)$ of bounded operators on some Hilbert space $H$. An abstract operator system is a matrix-ordered $*$-vector space with an Archimedean matrix order unit. The appropriate notion of isomorphism for operator systems is called a complete order isomorphism. In \cite{injectivity} it was shown that the two notions of operator systems are equivalent by proving that every abstract operator system is completely order isomorphic to a concrete operator system. In practice, when working with an abstract operator system, it is often convenient to have a concrete realization of this operator system as bounded operators on some Hilbert space. One can define the dual of a (finite-dimensional) operator system abstractly, however it is not always clear how this dual can be concretely realized.

A given (abstract) operator system can be embedded into $C^*$-algebras in different ways. Hamana proved the existence of a minimal $C^*$-algebra into which an operator system can be embedded in \cite{hamana}. This minimal $C^*$-algebra is called the $C^*$-envelope of the operator system. The difference between an operator system and its $C^*$-envelope is measured by the propagation number, which was first introduced in \cite{spectralTruncations}. 

In this paper we turn our attention towards a certain class of finite-dimensional operator systems. These operator systems appear as the non-commutative graphs that correspond to classical graphs and are termed ``graph operator systems'' in the context of quantum information theory \cite[Sect. 2.5]{graphOpSys}. Equivalently, such graphs correspond to tolerance relations on finite sets and this is the point of view we take here. As in \cite{toleranceRelations} we associate concrete operator systems to such tolerance relations. The $C^*$-envelopes and propagation numbers of operator systems associated with tolerance relations on finite sets are already known \cite{toleranceRelations}. The duals of these operator systems are briefly discussed in \cite[Sect. 3.5.1]{graphOpSys}, however, to the best of our knowledge, these duals have not yet been studied in depth. 

In this paper we study the duals of operator systems associated to tolerance relations on finite sets. In case the graph associated with such an operator system is chordal, we are able to find a concrete realization of its dual. In this case we will use this concrete realization to compute the $C^*$-envelope and propagation number of this dual. It will also become clear that characterizing the duals of operator systems associated to tolerance relations on finite sets is equivalent to a problem known as the positive semi-definite completion problem for partial matrices that is discussed in \cite{psd}. This problem has been well-studied and remains unsolved in full generality. Hence we cannot expect to easily obtain a fully general characterization of the dual of an operator system associated to a tolerance relation on a finite set whose corresponding graph might not be chordal.

From the point of view of \cite{toleranceRelations}, an operator system associated to a tolerance relation on a finite set is an approximation of some metric space. Hence one might expect the pure state space of this operator system to approximate this metric space in some sense. Since the pure states of an operator system correspond to extreme rays in its dual, this is motivation for us to determine the extreme rays of the dual of an operator system associated to a tolerance relation on a finite set. We determine these extreme rays in the case where the graph associated to this operator system is chordal, and find concordance with the general results in \cite[Proposition 3.11]{toleranceRelations} and \cite[Proposition 21]{ALL21}.  

Finally, we will illustrate our results by applying them to operator systems of band matrices. Band matrices are important because they frequently arise in applied linear algebra \cite{linAlg1, linAlg2}. Operator systems of band matrices are also interesting because they are exactly the operator systems that arise in Example 4.9 in \cite{toleranceRelations} when studying the interval at a finite resolution.

\section{Preliminaries}
\subsection{Preliminaries on operator systems}
In this subsection we briefly recall the theory of operator systems. 
More specifically, we review those objects related to an operator system that we are interested in computing, namely the dual of an operator system, the $C^*$-envelope of an operator system and the propagation number of an operator system. For more details we refer to \cite{injectivity, spectralTruncations, toleranceRelations, operatorSpaces, paulsen}.

\subsubsection{Abstract and concrete operator systems}
Let us first recap the definition of operator systems. We will note in particular the distinction between abstract and concrete operator systems, even though the two notions are equivalent.

Let $E$ be a \emph{$*$-vector space}, which means that it is a complex vector space equipped with a conjugate-linear involution $*:E\to E:v\mapsto v^*$. We write $E_h=\{x\in E\mid x^*=x\}$ for the real subspace of Hermitian elements. Let $E_+\subseteq E_h$ be a convex cone, which is proper in the sense that $E_+\cap -E_+=\{0\}$. We call the elements of $E_+$ \emph{positive} and we call a $*$-vector space equipped with such a cone an \emph{ordered} $*$-vector space. If $\varphi:E\to F$ is a linear map between ordered $*$-vector spaces then it is called \emph{positive} if $\varphi(E_+)\subseteq F_+$. An invertible positive map is called an \emph{order isomorphism} if its inverse is also positive. A positive map which is an order isomorphism onto its image is called an \emph{order embedding}.

If $E$ is a $*$-vector space, then the vector space $M_n(E)$ consisting of matrices with entries in $E$ is canonically a $*$-vector space by defining $(x_{ij})^*=(x^*_{ji})$, where $(x_{ij})\in M_n(E)$ denotes the matrix with entries $x_{ij}\in E$.
\begin{definition}
Let $E$ be a $*$-vector space where for each $n$ the space $M_n(E)_+$ is given an order such that the orders satisfy the following compatibility condition
\begin{equation}
    NM_m(E)_+N^* \subseteq M_n(E)_+,
\end{equation}
for every $N\in M_{nm}(\C)$. In this case $E$ is called \emph{matrix ordered}.
\end{definition}
A map $\varphi:E\to F$ between $*$-vector spaces induces maps $\varphi_n:M_n(E)\to M_n(F)$ by applying $\varphi$ entrywise. 
\begin{definition}
A linear map $\varphi:E\to F$ between matrix ordered $*$-vector spaces is called \emph{completely positive}, if $\varphi_n$ is positive for all $n$. If moreover $\varphi$ is invertible and $\varphi^{-1}$ is completely positive, then $\varphi$ is called a \emph{complete order isomorphism}.
\end{definition}
A map $\varphi$ between matrix ordered $*$-vector spaces is called a \emph{complete order embedding} if it is a complete order isomorphism onto its image.

If $e\in E_+$ is an element in an ordered $*$-vector space, then it is called an \emph{order unit} if for all $x\in E_h$ there exists $\lambda\geq 0$ such that $\lambda e+x\in E_+$. An order unit $e\in E$ is called \emph{Archimedean} if $
\lambda e+x\in E_+$ for all $\lambda>0$ implies $x\in E_+$.
\begin{definition}
An \emph{(abstract) operator system} is a matrix-ordered $*$-vector space for which there exists an \emph{Archimedean matrix order unit} $e\in E$, which means that
\[
I_n = \begin{pmatrix}e & 0 & \dots & 0\\ 0& \ddots & \ddots & \vdots \\ \vdots & \ddots & \ddots & 0 \\ 0 & \dots & 0 & e\end{pmatrix}
\]
is an Archimedean order unit for $M_n(E)$ for all $n$.
\end{definition}
Let $A$ be a $C^*$-algebra, then equipping $A$ with the cone $A_+$ of positive elements in $A$ turns $A$ into an ordered $*$-vector space. Using that $M_n(A)$ is a $C^*$-algebra for all $n$, we see that $A$ has a canonical matrix order. If $A$ has a unit, then $A$ is canonically an operator system.

We now turn our attention towards concrete operator systems and their relation with abstract operator systems.
\begin{definition}
Let $H$ be a Hilbert space and let $E\subseteq B(H)$ be a subspace. If $\text{id}_H\in E$ and $E=E^*$, then we call $E$ a \emph{(concrete) operator system}.
\end{definition}
If $E\subseteq B(H)$ is a concrete operator system, then by using that $M_n(B(H))\cong B(H^n)$ canonically we can give $E$ a matrix order by setting $M_n(E)_+=M_n(E)\cap B(H^n)_+$. Now $\text{id}_H\in E$ is an Archimedean matrix order unit and so $E$ is also an abstract operator system. The fact that every abstract operator system can be realized concretely in this way is a celebrated result by Choi and Effros.
\begin{theorem}[Choi-Effros]
Let $E$ be an abstract operator system with Archimedean matrix order unit $e$, then there exists a Hilbert space $H$ and a concrete operator system $E^\prime\subseteq B(H)$ together with a complete order isomorphism $\varphi:E\to E^\prime$ such that $\varphi(e)=\text{id}_H$.
\end{theorem}
\begin{proof}
The original proof can be found in \cite{injectivity}. See also Chapter 13 in \cite{paulsen}.
\end{proof}
\begin{remark}
Because every $C^*$-algebra $A$ can be embedded in the bounded operators on some Hilbert space $H$, every $*$-closed unital subspace $E\subseteq A$ of a unital $C^*$-algebra $A$ can be regarded as an operator system.
\end{remark}
While abstract operator systems are very useful for theoretical purposes, in practice it is often useful to have a concrete realization of the operator system one is working with.

In analogy with the $C^*$-algebra case we can define states on any operator system $E$. We denote by $E^d$ the dual vector space of $E$ consisting of linear functionals $\varphi:E\to\C$. We can give $E^d$ a canonical $*$-vector space structure by defining $\varphi^*(x)=\varphi(x^*)^*$ for all $x\in E$.
\begin{definition}
The \emph{state space} $\mathcal{S}(E)$ of an operator system $E$ consists of all unital positive linear functionals:
\begin{equation}
    \mathcal{S}(E) = \{\varphi\in E^d \mid \varphi(e)=1\text{ and }\varphi\text{ is positive}\}.
\end{equation}
It is clear that $\mathcal{S}(E)$ is a convex subset of $E^d$. Its extreme points are called \emph{pure states}.
\end{definition}
The importance of considering the pure states of an operator system is exemplified by the $C^*$-algebra case. If we consider a unital commutative $C^*$-algebra $C(X)$, which is the algebra of continuous functions on some compact Hausdorff space $X$, the pure states of $C(X)$ correspond to points of $X$.

\subsubsection{Dual operator systems}
For finite-dimensional operator systems there exists a satisfactory theory of dual operator systems, which was already developed in \cite{injectivity}. More recently the theory of duals for infinite-dimensional operator systems was also developed in \cite{infiniteDual}, however we will not make use of this infinite-dimensional theory. One of the uses of dual operator systems is to aid in determining the (pure) states of an operator system.

\begin{definition}
If $E$ is a matrix ordered $*$-vector space then we can define a matrix order on $E^d$ by exploiting the identification $M_n(E^d)\cong M_n(E)^d$:
\begin{equation}
    M_n(E^d)_+ = \{\varphi\in M_n(E)^d_h \mid \varphi(M_n(E)_+)\geq 0\}.
\end{equation}
With this matrix order $E^d$ is called the \emph{dual matrix ordered $*$-vector space} of $E$.
\end{definition}
In general $E^d$ does not contain an Archimedean matrix order unit making it into an operator system, even if $E$ is an operator system. However if $E$ is a finite-dimensional operator system then this problem does not arise.
\begin{proposition}[Choi-Effros]
If $E$ is a finite-dimensional operator system then the dual $E^d$ is an operator system.
\end{proposition}
\begin{proof}
This is Corollary 4.5 in \cite{injectivity}.
\end{proof}

It turns out that the operator system $M_n(\C)$ is actually self-dual and this fact will prove useful to us later.
\begin{lemma}\label{lem:identification}
The canonical vector space isomorphism $M_n(\C)^d\cong M_n(\C)$ is also a complete order isomorphism.
\end{lemma}
\begin{proof}
This is proved in \cite{injectivity} after Lemma 4.3.
\end{proof}

Let $E\subseteq A$ be an operator system sitting unitally inside some finite-dimensional $C^*$-algebra $A$, then its dual can be described as some kind of quotient. We write $E^\perp$ for the annihilator of $E$ inside $A^d$:
\begin{equation}
    E^\perp = \{\varphi\in A^d\mid \varphi(x)=0\text{ }\forall x\in E\}.
\end{equation}
We can equip the $*$-vector space $A^d/E^\perp$ with an order by
\[
(A^d/E^\perp)_+ = \{\varphi + E^\perp\in A^d/E^\perp\mid \varphi+\eta\in A^d_+\text{ for some }\eta\in E^\perp\}.
\]
The fact that this cone is indeed proper follows from the observation that $E^\perp$ contains no nonzero positive elements. Exploiting the canonical identification $M_n(A^d/E^\perp)\cong M_n(A^d)/M_n(E^\perp)$ we can extend this order to a matrix order. As the following theorem shows, this matrix ordered $*$-vector space is in fact an operator system.
\begin{remark}
The general theory of quotients of operator systems is quite subtle, see \cite{quotient}.
\end{remark}
\begin{theorem}\label{thm:quotient}
If $E\subseteq A$ is an operator system inside of a finite-dimensional $C^*$-algebra $A$, then we have a complete order isomorphism $A^d/E^\perp \cong E^d$, induced by restricting elements of $A^d$ to $E$.
\end{theorem}
\begin{proof}
This follows easily from Theorem 2.18 in \cite{toleranceRelations}.
\end{proof}
Combining Lemma \ref{lem:identification} and Theorem \ref{thm:quotient} we can give a description of the dual of an operator system $E\subseteq M_n(\C)$ sitting inside of the $C^*$-algebra $M_n(\C)$ of complex matrices.
\begin{corollary}\label{cor:dual}
Let $E\subseteq M_n(\C)$ be an operator system, then its dual $E^d$ can be identified with the $*$-vector space $E$ provided with the matrix order
\begin{equation}
    M_m(E^d)_+ = \{M\in M_m(E)\mid \exists N\in M_m(E^\perp), M+N\in M_{mn}(\C)_+\},
\end{equation}
where $E^\perp$ now denotes the orthogonal complement of $E$ in $M_n(\C)$ with respect to the Hilbert-Schmidt inner product.
\end{corollary}
\begin{proof}
Clear from the above.
\end{proof}

As said, one of the reasons for considering the dual of an operator system $E$, is to study the (pure) states of $E$. Recall that a state of $E$ is an element $\varphi\in E^d_+$ such that $\varphi(e)=1$. If $\varphi\in E^d_+$ is such that $\varphi(e)=0$, then because $e$ is an order unit it follows that $\varphi=0$ and hence every nonzero element of $E^d_+$ can be rescaled to be a state. This means that there is a correspondence between states of $E$ and \emph{rays} in $E^d_+$, which is to say half lines $R=\{\lambda \varphi \mid \lambda\geq 0\}$ with $0\neq \varphi\in E^d_+$. A ray $R$ in $E^d_+$ is called \emph{extremal} if $\varphi_1,\varphi_2\in E^d_+$ and $\varphi_1+\varphi_2\in R$ imply that $\varphi_1,\varphi_2\in R$. Pure states of $E$ correspond to extremal rays in $E^d_+$. 

\subsubsection{\texorpdfstring{$C^*$}{C*}-envelopes and propagation numbers}
Any (abstract) operator system $E$ can be embedded in a $C^*$-algebra and one can consider the $C^*$-algebra generated by $E$ under this embedding. However there are multiple ways to embed $E$ into a $C^*$-algebra, which lead to different $C^*$-algebras generated by $E$. The $C^*$-envelope of $E$ is the unique minimal $C^*$-algebra generated by $E$. We can use the $C^*$-envelope of an operator system to associate an invariant to this operator system which is called the propagation number.

If $E$ is an operator system and $\iota: E\to A$ is a complete order embedding into a $C^*$-algebra $A$, we write $C^*(\iota(E))$ for the $C^*$-algebra generated by the image $\iota(E)$. A \emph{$C^*$-extension} of an operator system is a complete order embedding $\iota: E\to A$ into a $C^*$-algebra $A$ such that $C^*(\iota(E))=A$.
\begin{definition}
Let $E$ be an operator system and $\iota:E\to A$ a $C^*$-extension such that for any other $C^*$-extension $\kappa:E\to B$ there exists a unique $*$-homomorphism $\psi:B\to A$ satisfying $\psi\circ \kappa=\iota$. In this case $A$ is called the \emph{$C^*$-envelope} of $E$ and we write $A=\Cenv(E)$.
\end{definition}
\begin{remark}
Note that the map $\psi$ above has to be surjective.
\end{remark}
By this universal property $C^*$-envelopes are clearly unique. Their existence was first established by Hamana.
\begin{theorem}[Hamana]
If $E$ is an operator system, then it has a $C^*$-envelope: $\iota: E\to \Cenv(E)$.
\end{theorem}
\begin{proof}
This was first proven by Hamana in \cite{hamana}.
\end{proof}
As mentioned before, using the $C^*$-envelope we can define an invariant of operator systems which is called the propagation number. The propagation number was first introduced in \cite{spectralTruncations}. It can be useful to distinguish different operator systems and it measures the difference between an operator system and its $C^*$-envelope.

If $E\subseteq A$ is an operator system sitting inside of a $C^*$-algebra $A$, then we write $E^{\circ n}$ for the closed linear span of products of elements in $E$ with at most $n$ factors.

\begin{definition}
Let $E$ be an operator system and $\iota:E\to \Cenv(E)$ its $C^*$-envelope. The \emph{propagation number} $\prop(E)$ of $E$ is defined by
\begin{equation}
    \prop(E) = \inf\{n\in\N\mid \iota(E)^{\circ n}=\Cenv(E)\}.
\end{equation}
\end{definition}
\begin{remark}
Note that the propagation number of an operator system can be infinite.
\end{remark}
The universal property of $C^*$-envelopes shows that the propagation number is well-defined (i.e. independent of the choice of $C^*$-envelope). It also follows from the universal property that the propagation number is invariant under complete order isomorphisms. In fact something even stronger is true. Proposition 2.42 in \cite{spectralTruncations} states that the propagation number is invariant under stable equivalence of operator systems. Moreover, a notion of Morita equivalence for operator systems was developed recently in \cite{morita} and it was shown there that this notion coincides with stable equivalence. Hence we may say that the propagation number is Morita invariant as well.

\subsection{Preliminaries on partial matrices and chordal graphs}
To prove our main result we require a result on the positive semi-definite completion problem for partial Hermitian matrices. Therefore in this subsection we will review the relevant definitions and state this key result on the completion problem for partial Hermitian matrices. We will also consider chordal graphs, as we need them to formulate the key result. For a more detailed discussion consult \cite{completion}.

A \emph{partial} $n\times n$ matrix $M=(m_{ij})$ is a complex $n\times n$ matrix some of whose entries $m_{ij}$ might be unspecified. We will assume that all partial matrices that we work with are partial Hermitian matrices, which means that $m_{ii}$ is defined and real for all $1\leq i\leq n$ and if $m_{ij}$ is defined, then $m_{ji}$ is defined and $m_{ji}=\overline{m_{ij}}$. If $M=(m_{ij})$ is a partial matrix, then its graph $G(M)=(X,\mathcal{E})$ has vertices $X=\{1,\dots,n\}$ and the set of edges $\mathcal{E}$ is such that there is an edge between $i$ and $j$ if and only if $m_{ij}$ is specified (since we work with partial Hermitian matrices, we need only consider undirected graphs).

When given a partial matrix, it is interesting to look for completions of said matrix that have certain properties. A \emph{completion} of a partial matrix $(m_{ij})$ is a matrix $N=(n_{ij})\in M_n(\C)$ such that $n_{ij}=m_{ij}$ whenever $m_{ij}$ is specified. This completion is called \emph{positive} if $N$ is a positive matrix (when we say ``positive matrix" we will always mean ``positive semi-definite matrix").

Because principal submatrices of positive matrices are again positive, a necessary condition for a partial matrix $(m_{ij})$ to admit a positive completion is that the principal submatrix $(m_{ij})_{i,j\in C}$ is positive for every $C\subseteq X$ such that all the $m_{ij}$ with $i,j\in C$ are specified (i.e. such that $C\subseteq X$ forms a clique in the graph $G(M)$). 
\begin{definition}
A partial matrix $M=(m_{ij})$ is called \emph{partially positive} if every fully specified principal submatrix $(m_{ij})_{i,j\in C}$ with $C\subseteq X$ is positive. Said differently, $M$ is partially positive if $(m_{ij})_{i,j\in C}$ is positive for every clique $C\subseteq X$.
\end{definition}
\begin{remark}
It suffices to verify that $(m_{ij})_{i,j\in C}$ is positive for every maximal clique $C\subseteq X$.
\end{remark}
It turns out that the question of whether being partially positive is also a sufficient condition for $M$ to admit a positive completion depends on the graph $G(M)$.
\begin{definition}
A graph $G=(X,\mathcal{E})$ is called \emph{chordal} if whenever we have a cycle in $G$ of length at least $4$, then there exists a chord in this cycle (i.e. an edge in $\mathcal{E}$ connecting two vertices which are not adjacent in the cycle).
\end{definition}
We can now state the key result that we need to prove our main result.
\begin{theorem}\label{thm:chordal}
A graph $G$ is chordal if and only if every partially positive matrix $M$ with graph $G(M)=G$ admits a positive completion.
\end{theorem}
\begin{proof}
This is Theorem 7 in \cite{completion}.
\end{proof}
There are some results that show when a partially positive matrix whose graph is not chordal admits a positive completion, however these are not applicable in full generality and a general characterization of partially positive matrices that admit a positive completion is not known (see \cite{psd} and \cite{schur}).

A characterization of chordal graphs that we shall need later to compute propagation numbers is the following.
\begin{definition}
If $G=(X,\mathcal{E})$ is a graph a \emph{perfect elimination ordering} of its vertices is an ordering $v_1<v_2<\dots<v_n$ of all the vertices $v_i\in X$ such that $C[i]=\{v_j\mid j=i\text{ or }j>i\text{ and }\{i,j\}\in \mathcal{E}\}$ is a clique for all $i$.
\end{definition}
\begin{proposition}\label{prop:elimination}
A graph $G$ is chordal if and only if there exists a perfect elimination ordering of its vertices.
\end{proposition}
\begin{proof}
This is Theorem 1 in \cite{chordal}.
\end{proof}

\section{Operator systems associated with finite tolerance relations}
In this section we prove our main result, which is a concrete realization for the duals of a certain class of finite-dimensional operator systems. We can then use this realization to compute the extremal rays, $C^*$-envelopes and propagation numbers of these duals.

Let $X$ be a finite set and let $\mathcal{R}\subseteq X\times X$ be a reflexive and symmetric relation. Such a relation is called a \emph{tolerance relation}. A finite set with a tolerance relation $\mathcal{R}$ can be interpreted as a graph which we denote by $G(\mathcal{R})$. A tolerance relation $\mathcal{R}$ on $X$ gives rise to a finite-dimensional operator system $E(\mathcal{R})\subseteq M_n(\C)$ as follows. We may assume without loss of generality that $X=\{1,\dots,n\}$ and then we define
\begin{equation}
    E(\mathcal{R}) = \{(x_{ij})\in M_n(\C)\mid x_{ij}=0\text{ if }(i,j)\notin \mathcal{R}\}.
\end{equation}

These operator systems are the operator systems associated to tolerance relations on finite sets that are considered in \cite{toleranceRelations}. The $C^*$-envelopes and propagation numbers of these operator systems are known. Indeed, a slight generalization of Proposition 3.7 in \cite{toleranceRelations} states that
\begin{equation}
    \Cenv(E(\mathcal{R})) = \bigoplus_{K\in\mathcal{K}} M_{|K|}(\C),
\end{equation}
where $\mathcal{K}$ is the set of equivalence classes of the equivalence relation generated by $\mathcal{R}$. Moreover $\prop(E(\mathcal{R}))=\diam(G(\mathcal{R}))$, where $\diam(G(\mathcal{R}))$ denotes the graph diameter. Interesting examples of operator systems associated to tolerance relations on finite sets are furnished by operator systems for finite partial partitions, which are discussed in Subsection 4.1 of \cite{toleranceRelations}. 

Since the operator system $E(\mathcal{R})$ is already fairly well understood, the next logical step is to study the dual $E(\mathcal{R})^d$. We would like to find a concrete realization for $E(\mathcal{R})^d$. In general this is a difficult problem, however if the graph $G(\mathcal{R})$ is chordal we will see that one can find a concrete realization. We then use this concrete realization to determine the $C^*$-envelope and the propagation number of $E(\mathcal{R})^d$ if $G(\mathcal{R})$ is chordal.

\subsection{A concrete realization}
If $\mathcal{R}$ is a tolerance relation on a finite set $X$, the space $M_n(E(\mathcal{R}))$ is naturally an operator system and this operator system again arises from a tolerance relation.
\begin{definition}
    If $\mathcal{R}$ is a tolerance relation on the finite set $X$, let $\mathcal{R}^n$ denote the tolerance relation on $\widetilde{X}=X\times\{1,\dots, n\}$ given by 
    \begin{equation}
        \mathcal{R}^n = \{((x,i),(y,j))\in \widetilde{X}\times \widetilde{X}\mid (x,y)\in\mathcal{R}\}.
    \end{equation}
\end{definition}
It is clear that we have a canonical identification $M_n(E(\mathcal{R}))\cong E(\mathcal{R}^n)$.
The set of maximal cliques of the graph $G(\mathcal{R})$ is denoted by $\mathcal{C}$.
\begin{lemma}\label{lem:cliques}
Let $n\in\N$ be given and let $\widetilde{\mathcal{C}}$ denote the set of maximal cliques in the graph $G(\mathcal{R}^n)$, then there is a canonical bijection $\theta: \mathcal{C}\to\widetilde{\mathcal{C}}$ given by $\theta(C) = \{ (x,i)\in X\times\{1,\dots,n\}\mid x\in C\}$.
\end{lemma}
\begin{proof}
It is clear that $\theta$ is well-defined and injective. To see that it is surjective, let $\widetilde{C}\in\widetilde{\mathcal{C}}$. Now let $C^\prime=\{x\in X\mid \exists i, (x,i)\in\widetilde{C}\}$, then it is clear that $C^\prime$ is a clique in $X$ and hence it is contained in some maximal clique $C\in\mathcal{C}$. One easily sees that $\widetilde{C}\subseteq \theta(C)$ and hence by maximality of $\widetilde{C}$ we must have $\widetilde{C}=\theta(C)$ and we conclude that $\theta$ is surjective.
\end{proof}

We use the canonical vector space identification $E(\mathcal{R})^d\cong E(\mathcal{R})\subseteq M_n(\C)$ to view $E(\mathcal{R})^d$ as a subspace of $M_n(\C)$. 
\begin{theorem}\label{thm:main}
Let $\mathcal{R}$ be a tolerance relation on a finite set $X$ with chordal graph $G(\mathcal{R})$. Then the map 
\begin{equation}
    \Phi: E(\mathcal{R})^d\to \bigoplus_{C\in \mathcal{C}} M_{|C|}(\C):(x_{ij})\mapsto ((x_{ij})_{i,j\in C})_{C\in\mathcal{C}}
\end{equation}
is a complete order embedding.
\end{theorem}
\begin{proof}
Let us first prove that $\Phi$ is an order embedding. Note that Corollary \ref{cor:dual} implies that 
\begin{equation}\label{eq:thm2}
    E(\mathcal{R})^d_+ = \{ M\in E(\mathcal{R})^d \mid \exists N\in E(\mathcal{R})^\perp,\: M+N\in M_n(\C)_+\}.
\end{equation}
Let $(x_{ij})\in E(\mathcal{R})^d_+$, then there exists $(y_{ij})\in E(\mathcal{R})^\perp$ such that $(x_{ij}+y_{ij})$ is a positive matrix. Now if $C\in\mathcal{C}$, then the principal submatrix $(x_{ij})_{i,j\in C}$ is the same as the principal submatrix $(x_{ij}+y_{ij})_{i,j\in C}$, since $y_{ij}=0$ if $ij$ is an edge in $G(\mathcal{R})$. This implies that $(x_{ij})_{i,j\in C}$ is a principal submatrix of a positive matrix and hence positive. We conclude that $\Phi$ is positive.

It is clear that $\Phi$ is injective. To see that it is also an order embedding, observe that we can interpret $E(\mathcal{R})^d$ as consisting of partial matrices with graph $G(\mathcal{R})$. From this perspective, if we let $(x_{ij})\in E(\mathcal{R})^d$, then $\Phi((x_{ij}))$ is positive precisely when $(x_{ij})$ is partially positive. Since $G(\mathcal{R})$ is chordal, we can use Theorem \ref{thm:chordal} to see that if $\Phi((x_{ij}))$ is positive then the partial matrix $(x_{ij})$ must have a positive completion, which can be written as $(x_{ij})+(y_{ij})$ for some $(y_{ij})\in E(\mathcal{R})^\perp$. By Equation \ref{eq:thm2} this means that in this case $(x_{ij})$ must have been positive. This shows that $\Phi$ is an order embedding.

To see that $\Phi$ is in fact a complete order embedding, let $n\in \N$. Then we have to show that 
\begin{equation}
    \Phi_n: M_n(E(\mathcal{R})^d)\to M_n\left(\bigoplus_{C\in \mathcal{C}} M_{|C|}(\C)\right)
\end{equation}
is an order embedding. We have the identifications $M_n(E(\mathcal{R})^d)\cong M_n(E(\mathcal{R}))^d\cong E(\mathcal{R}^n)^d$. Using Lemma \ref{lem:cliques} we also have the identifications
\begin{equation}
    M_n\left(\bigoplus_{C\in \mathcal{C}} M_{|C|}(\C)\right) \cong \bigoplus_{C\in \mathcal{C}} M_n(M_{|C|}(\C)) \cong \bigoplus_{C\in \mathcal{C}} M_{n|C|}(\C) \cong \bigoplus_{\widetilde{C}\in \widetilde{\mathcal{C}}} M_{|\widetilde{C}|}(\C).
\end{equation}
Combining this, we see that we can view $\Phi_n$ as a map $\widetilde{\Phi}:E(\mathcal{R}^n)^d \to \bigoplus_{\widetilde{C}\in \widetilde{\mathcal{C}}} M_{|\widetilde{C}|}(\C)$. This $\widetilde{\Phi}$ is of the same form as the map $\Phi$ and hence by the arguments in the previous two paragraphs it is an order embedding.
\end{proof}
\begin{remark}
Equation \ref{eq:thm2} shows that determining the dual of an operator system associated to a tolerance relation on a finite set is equivalent to a positive semi-definite completion problem for certain partial matrices as described in \cite{psd}. Since this problem remains unsolved in full generality despite receiving significant attention, it is unlikely that we will easily be able to describe the dual of an arbitrary operator system associated to a tolerance relation on a finite set.
\end{remark}

\begin{remark}
Note that even though we choose to identify the vector spaces underlying the operator systems $E(\mathcal{R})$ and $E(\mathcal{R})^d$, the associated cones of positive elements are quite different. The cone $E(\mathcal{R})_+$ consists of matrices which are positive semi-definite in the classical linear algebraic sense, while we have just seen that $E(\mathcal{R})^d_+$ consists of partial matrices which are partially positive.
\end{remark}

\subsection{Extremal rays}
If the graph $G(\mathcal{R})$ is chordal, then using the description of the order on $E(\mathcal{R})^d$ given by Equation~\ref{eq:thm2}, we can determine the extremal rays of $E(\mathcal{R})^d_+$ and hence the pure states of $E(\mathcal{R})$. Let us write $\pi:M_n(\C)\to E^d$ for the orthogonal projection (with respect to the Hilbert-Schmidt inner product). Because of Equation \ref{eq:thm2} it is clear that $E^d_+=\pi(M_n(\C)_+)$ if $G(\mathcal{R})$ is chordal. In this case we will call elements of $E^d_+$ partially positive matrices, since they can be interpreted as partial matrices with graph $G(\mathcal{R})$ which are partially positive.

\begin{lemma}\label{lem:extremal}
The extremal rays of $M_n(\C)_+$ are precisely those generated by rank one matrices.
\end{lemma}
\begin{proof}
This is well known.
\end{proof}

\begin{theorem}\label{thm:extremal}
If $\mathcal{R}$ is a tolerance relation on a finite set $X$ with chordal graph $G(\mathcal{R})$, then the extremal rays of $E^d_+$ are generated by those $0\neq (x_{ij})\in E^d_+$ that satisfy: 
\begin{enumerate}
    \item for every maximal clique $C\in\mathcal{C}$ the matrix $(x_{ij})_{i,j\in C}$ is of rank one or zero, 
    \item the induced subgraph $H$ of $G(\mathcal{R})$, induced by the vertices $i\in X$ such that $x_{ii}\neq 0$, is connected.
\end{enumerate} 
\end{theorem}
\begin{proof}
This proof is inspired by the proof of Theorem 3.2 in \cite{extremal}, where we correct for a slight gap in the proof. Assume $0\neq (x_{ij})\in E^d_+$ generates an extremal ray. Pick a positive $(y_{ij})\in \pi^{-1}((x_{ij}))$. There exist $(y^1_{ij}),(y^2_{ij})\in M_n(\C)_+$ such that $(y^1_{ij})$ is of rank one, $\pi((y^1_{ij}))\neq 0$ and $(y_{ij})=(y_{ij}^1+y_{ij}^2)$. Now clearly $(x_{ij})=\pi((y^1_{ij}))+\pi((y^2_{ij}))$ and since $(x_{ij})$ generates an extremal ray, this implies that $(x_{ij})$ is a multiple of $\pi((y^1_{ij}))$. Let $C\in\mathcal{C}$, then $(x_{ij})_{i,j\in C}$ is proportional to the principal submatrix $(y^1_{ij})_{i,j\in C}$ of a rank one matrix and hence has rank one or zero.

Suppose $H$ is not connected and decompose $H$ into two nonempty subsets $H=J_1\cup J_2$ such that there is no edge in $H$ from a vertex in $J_1$ to a vertex in $J_2$. Define $(x^1_{ij}),(x^2_{ij})\in E^d_+$ by $x_{ij}^k=x_{ij}$ if $i,j\in J_k$ and $x_{ij}^k=0$ otherwise, then $(x^1_{ij})$ and $(x^2_{ij})$ are clearly not multiples of $(x_{ij})$. Note that $x_{kl}=0$ unless $k,l\in H$, because $(x_{ij})$ is partially positive. Indeed, if $k\notin H$ (the case $l\notin H$ is similar) then $x_{kk}=0$ and by picking a maximal clique $C$ containing $k$ and $l$, we find that $x_{kl}$ is an entry in a positive matrix $(x_{ij})_{i,j\in C}$ whose diagonal entry $x_{kk}$ is zero and hence we must have that $x_{kl}=0$. Also note that if $k\in J_1$ and $l\in J_2$ or vice versa, then $(k,l)\notin\mathcal{R}$ so $x_{kl}=0$. It follows that $(x_{ij})=(x^1_{ij})+(x^2_{ij})$, contradicting that $(x_{ij})$ generates an extremal ray. We conclude that $H$ must be connected.

Now suppose $0\neq (x_{ij})\in E^d_+$ satisfies the two conditions of the theorem and assume $(x_{ij})=(x^1_{ij})+(x^2_{ij})$ for some $(x^1_{ij}),(x^2_{ij})\in E^d_+$. We will prove that $(x^1_{ij})$ is a multiple of $(x_{ij})$, proving that $(x_{ij})$ generates an extremal ray. Clearly there exist scalars $\lambda_{i}$ such that $x^1_{ii}=\lambda_i x_{ii}$ for all $i$. Indeed if $x_{ii}=0$ then $x^1_{ii}=0$, because $x^1_{ii},x^2_{ii}\geq 0$. If $kl$ is an edge in $H$, let $C\in\mathcal{C}$ be a maximal clique containing $k$ and $l$. By assumption the matrix $(x_{ij})_{i,j\in C}=(x_{ij}^1)_{i,j\in C}+(x_{ij}^2)_{i,j\in C}$ is of rank (at most) one and since $(x_{ij}^1)_{i,j\in C}$ and $(x_{ij}^2)_{i,j\in C}$ are positive ($(x_{ij}^1)$ and $(x_{ij}^2)$ are partially positive) using Lemma \ref{lem:extremal} it follows that $(x_{ij}^1)_{i,j\in C}$ is a multiple of $(x_{ij})_{i,j\in C}$, which implies that $\lambda_k=\lambda_l$ and $\lambda_k(x_{ij})_{i,j\in C}=(x_{ij}^1)_{i,j\in C}$. Since we have assumed $H$ to be connected, this implies that $\lambda_i=\lambda_j$ for all vertices $i$ and $j$ in $H$ and we now denote this scalar by $\lambda$. Now let $(k,l)\in \mathcal{R}$ then $x_{kl}^1=x_{kl}=0$ unless $k$ and $l$ are vertices in $H$, because $(x_{ij})$ and $(x_{ij}^1)$ are partially positive. If $k,l\in H$, then by picking a maximal clique $C\in\mathcal{C}$ containing both $k$ and $l$ and using that $(x_{ij}^1)_{i,j\in C}=\lambda(x_{ij})_{i,j\in C}$ we see that $x_{kl}^1=\lambda x_{kl}$. We conclude that $(x_{ij}^1)=\lambda(x_{ij})$, which proves that $(x_{ij})$ generates an extremal ray.
\end{proof}
\begin{remark}\label{rmk:projection}
This proof also shows that if $0\neq (x_{ij})\in E^d_+$ generates an extremal ray, then it is the projection of some positive rank one matrix $(y_{ij})\in M_n(\C)_+$.
\end{remark}
\begin{remark}
Note that this theorem is essentially a reformulation of Proposition 3.11 in \cite{toleranceRelations}. A similar result has been obtained in \cite[Proposition 27]{ALL21}.
\end{remark}

\subsection{The \texorpdfstring{$C^*$}{C*}-envelope}
In this subsection, we use our concrete realization for the dual $E(\mathcal{R})^d$ for tolerance relations $\mathcal{R}$ with chordal graphs $G(\mathcal{R})$, to compute the $C^*$-envelope $\Cenv(E(\mathcal{R})^d)$.
\begin{lemma}\label{lem:C*}
If $\mathcal{R}$ is a tolerance relation on a finite set $X$ with chordal graph $G(\mathcal{R})$ and \sloppy{${\Phi: E(\mathcal{R})^d\to \bigoplus_{C\in \mathcal{C}} M_{|C|}(\C)}$} is the complete order embedding of Theorem \ref{thm:main}, then 
\begin{equation}
    C^*(\Phi(E(\mathcal{R})^d))=\bigoplus_{C\in \mathcal{C}} M_{|C|}(\C).
\end{equation}
\end{lemma}
\begin{proof}
Because $C^*(\Phi(E(\mathcal{R})^d))$ is a finite-dimensional $C^*$-algebra we have
\begin{equation}
    C^*(\Phi(E(\mathcal{R})^d)) \cong \bigoplus_{i=1}^N M_{n_i}(\C)
\end{equation}
for some integers $N$ and $n_i$. This implies that we have an injective unital $*$-homomorphism
\begin{equation}
    \gamma: \bigoplus_{i=1}^N M_{n_i}(\C)\to \bigoplus_{C\in \mathcal{C}} M_{|C|}(\C).
\end{equation}
Let $C\in\mathcal{C}$. Composing $\gamma$ with the projection $\pi_C:\bigoplus_{C\in \mathcal{C}} M_{|C|}(\C)\to M_{|C|}(\C)$, we obtain a ${*\text{-homomorphism}}$ $\tau_C:\bigoplus_{i=1}^N M_{n_i}(\C)\to M_{|C|}(\C)$. Since 
\begin{equation}
    \Phi(E(\mathcal{R})^d)\subseteq C^*(\Phi(E(\mathcal{R})^d))=\gamma\left(\bigoplus_{i=1}^N M_{n_i}(\C)\right)
\end{equation} 
and $\pi_C(\Phi(E(\mathcal{R})^d))=M_{|C|}(\C)$, we see that $\tau_C$ must be surjective. The general structure of unital $*$-homomorphisms between finite-dimensional $C^*$-algebras implies the existence of nonnegative integers $(a_{Ci})_{C\in\mathcal{C},\:  1\leq i\leq N}$ such that $|C|=\sum_{i=1}^n a_{Ci}n_{i}$ and $\tau_C\circ\gamma$ is unitarily equivalent to the direct sum $\id_{M_{n_{1}}(\C)}^{a_{C1}}\oplus \dots \oplus \id_{M_{n_{N}}(\C)}^{a_{CN}}$ for all $C\in\mathcal{C}$ (see Lemma III.2.1 in \cite{byExample}). Let $C\in\mathcal{C}$, then since $\tau_C$ is surjective we find that $\sum_{i=1}^n a_{Ci}=1$ and thus $a_{Ci}=1$ and $n_i=|C|$ for some $i$. Since $\gamma$ is injective, this implies that $N\leq |\mathcal{C}|$.

Suppose $N< |\mathcal{C}|$, then we can find $C,D\in\mathcal{C}$ and $k$ such that $C\neq D$, $a_{Ck}=a_{Dk}=1$ and $a_{Cj}=a_{Dj}=0$ for all $j\neq k$. This implies that $(\tau_C\oplus \tau_D)\circ \gamma$ has $n_k^2$-dimensional image, since 
\begin{equation}
((\tau_C\oplus \tau_D)\circ \gamma)\left(\bigoplus_{i=1}^N M_{n_i}(\C)\right)=((\tau_C\oplus \tau_D)\circ \gamma)(M_{n_k}(\C)). 
\end{equation}
However, upon examination of the image of $\Phi$ we see that $(\tau_C\oplus \tau_D)(\Phi(E(\mathcal{R})^d))$ is a subspace of the image of $(\tau_C\oplus \tau_D)\circ \gamma$ with greater dimension. This contradiction shows that we must have $N=|\mathcal{C}|$. Hence there is a bijective correspondence between the $n_i$ and the $C\in\mathcal{C}$, from which we deduce that the dimensions of $\bigoplus_{i=1}^N M_{n_i}(\C)$ and $\bigoplus_{C\in \mathcal{C}} M_{|C|}(\C)$ agree. This proves that $C^*(\Phi(E(\mathcal{R})^d)=\bigoplus_{C\in \mathcal{C}} M_{|C|}(\C)$.

\end{proof}
\begin{theorem}\label{thm:C*env}
If $\mathcal{R}$ is a tolerance relation on a finite set $X$ with chordal graph $G(\mathcal{R})$, then we have 
\begin{equation}
    \Cenv(E(\mathcal{R})^d)\cong \bigoplus_{C\in \mathcal{C}} M_{|C|}(\C).
\end{equation}
\end{theorem}
\begin{proof}
By the universal property of $C^*$-envelopes and Lemma \ref{lem:C*}, there must exist a surjective $*$-homomorphism
\begin{equation}
    \psi: \bigoplus_{C\in \mathcal{C}} M_{|C|}(\C)\to \Cenv(E(\mathcal{R})^d),
\end{equation}
such that $\psi\circ \Phi=\iota$, where $\iota:E(\mathcal{R}^d)\to \Cenv(E(\mathcal{R})^d)$ is a complete order embedding.
If $\psi$ is injective then we are done, so assume that $\psi$ is not injective. Then there must exist an isomorphism
\begin{equation}
    \xi:\Cenv(E(\mathcal{R})^d) \xrightarrow{\sim} \bigoplus_{C\in \mathcal{C}^\prime} M_{|C|}(\C)
\end{equation}
for some subset $\mathcal{C}^\prime\subsetneq \mathcal{C}$, such that $\xi\circ \iota=\Phi^\prime$, where $\Phi^\prime$ is given by
\begin{equation}
    \Phi^\prime: E(\mathcal{R})^d\to \bigoplus_{C\in \mathcal{C}^\prime} M_{|C|}(\C):(x_{ij})\mapsto ((x_{ij})_{i,j\in C})_{C\in\mathcal{C}^\prime}.
\end{equation}
As a composition of complete order embeddings, $\Phi^\prime$ is a complete order embedding. Pick $C\in\mathcal{C}\setminus \mathcal{C}^\prime$ and consider the operator system $F\subseteq E(\mathcal{R})^d$ where
\begin{equation}
    F=\{(x_{ij})\in E(\mathcal{R})\mid x_{ij}=0\text{ unless }i,j\in C\}.
\end{equation}
Clearly we have a canonical vector space identification $M_{|C|}(\C)\cong F$ and using Equation \ref{eq:thm2} it is not hard to see that this is also a complete order isomorphism. This implies that $\Cenv(F)=M_{|C|}(\C)$. Given $D\in \mathcal{C}^\prime$ let us denote by $\pi_D$ the projection $\pi_D:\bigoplus_{C\in \mathcal{C}^\prime} M_{|C|}(\C)\to M_{|D|}(\C)$. Using the definition of the map $\Phi^\prime$ we see that for all $D\in\mathcal{C}^\prime$ we have that $C^*(\pi_D(\Phi^\prime(F)))$ is isomorphic to a matrix algebra $M_n(\C)$ with $n<|C|$. This implies that $C^*(\Phi^\prime(F))$ is isomorphic to a direct sum of matrix algebras of the form $M_n(\C)$ with $n<|C|$. We conclude that $\Cenv(F)\cong M_{|C|}(\C)$ is a quotient of this direct sum of matrix algebras, however this is a contradiction. This proves that $\psi$ must have been injective after all.
\end{proof}

\begin{remark}
Note that we have in fact shown that the map $\Phi$ from Theorem \ref{thm:main} realizes the $C^*$-envelope of $E(\mathcal{R})$.
\end{remark}

\subsection{The propagation number}
Now that we know the $C^*$-envelope of $E(\mathcal{R})^d$ when $G(\mathcal{R})$ is chordal, we can compute its propagation number. To aid in doing so, let us develop some notation first.

If we assume that the graph $G(\mathcal{R})$ is chordal, then by Proposition \ref{prop:elimination} there exists a perfect elimination ordering for its vertices. Without loss of generality we will assume that $X=\{1,\dots,n\}$ and the usual ordering of the integers constitutes a perfect elimination ordering for $X$. If $C\in\mathcal{C}$ is a maximal clique, let $i\in C$ be the minimal element, then it is easy to see that $C=C[i]$. Let $i,j\in X$, then since every maximal clique is of the form $C[i]$, there exists a subset $\mathcal{I}_{ij}\subseteq X$, such that $\{C[k]\mid k\in\mathcal{I}_{ij}\}$ is precisely the set of maximal cliques that contain the edge $ij$ (if $i=j$, we take $\mathcal{I}_{ii}$ to be such that we obtain the set of all maximal cliques that contain the vertex $i$). Clearly $\mathcal{I}_{ij}$ is bounded from above by $\min(i,j)$.

Write $E_{ij}$ for the matrix units in $M_n(\C)$, then $E_{ij}\in E(\mathcal{R})^d\subseteq M_n(\C)$ if and only if $ij$ is an edge in $G(\mathcal{R})$. If $C\in\mathcal{C}$, then we index the entries of matrices in $M_{|C|}(\C)$ by $C$. Now given $C\in\mathcal{C}$ and $i,j\in X$ let the element $E^C_{ij}$ of $\bigoplus_{C\in\mathcal{C}}M_{|C|}(\C)$ be defined as follows. If $i\notin C$ or $j\notin C$, then $E_{ij}^C=0$. Otherwise, $E_{ij}^C$ is the element whose $M_{|D|}(\C)$ component is zero for all $D\in\mathcal{C}$ with $D\neq C$ and whose $M_{|C|}(\C)$ component is equal to the matrix unit $E_{ij}$ which has a one in the entry indexed by $ij$ and zeroes everywhere else. It follows that $\Phi(E_{ij})=\sum_{C\in\mathcal{C}} E_{ij}^C$. If $C=C[k]$ is a maximal clique, we also write $E^k_{ij}$ instead of $E^C_{ij}$.

\begin{theorem}\label{thm:prop}
Let $\mathcal{R}$ be a tolerance relation on a finite set $X$ with chordal graph $G(\mathcal{R})$. If the relation $\mathcal{R}$ is an equivalence relation, then $\prop(E(\mathcal{R})^d)=1$. Otherwise $\prop(E(\mathcal{R})^d)=2$.
\end{theorem}
\begin{proof}
If $\mathcal{R}$ is an equivalence relation, then the maximal cliques of $G(\mathcal{R})$ are disjoint and it is easy to see that $\Phi:E(\mathcal{R})^d\to \bigoplus_{C\in\mathcal{C}}M_{|C|}(\C)$ is surjective. Since $\Cenv(E(\mathcal{R})^d)=\bigoplus_{C\in\mathcal{C}}M_{|C|}(\C)$, this means that $\prop(E(\mathcal{R})^d)=1$.

If $\mathcal{R}$ is not an equivalence relation, then there are some cliques in $\mathcal{C}$ which are not disjoint and it is easy to see that $\Phi$ is not surjective and hence $\prop(E(\mathcal{R})^d)>1$. It now suffices to show that every element of $\bigoplus_{C\in\mathcal{C}}M_{|C|}(\C)$ can be written as a linear combination of products of two elements in $\Phi(E(\mathcal{R})^d)$. Since $\bigoplus_{C\in\mathcal{C}}M_{|C|}(\C)$ is spanned by the elements of the form $E_{ij}^C$ it suffices to prove that the $E_{ij}^C$ can be written as such linear combinations. Let $i,j\in X$ be given. We will prove this by induction.

Let $k\in\mathcal{I}_{ij}$ be minimal, then 
\begin{equation}
    \Phi(E_{ik})\Phi(E_{kj})=\left(\sum_{C\in\mathcal{C}} E_{ik}^C\right)\left(\sum_{D\in\mathcal{C}} E_{kj}^D\right).
\end{equation}
Because $E_{ik}^C$ and $E_{kj}^D$ live in different components of $\bigoplus_{C\in\mathcal{C}}M_{|C|}(\C)$ if $C\neq D$, no cross terms appear in this product and
\begin{equation}
    \Phi(E_{ik})\Phi(E_{kj})=\sum_{C\in\mathcal{C}} E_{ik}^CE_{kj}^C.
\end{equation}
Since $k\in\mathcal{I}_{ij}$ is minimal, $C[k]$ is the only maximal clique containing $i,j$ and $k$ and thus this sum reduces to the single term $E^k_{ik}E_{kj}^k=E_{ij}^k$ and we find that
\begin{equation}
    E_{ij}^k=\Phi(E_{ik})\Phi(E_{kj}).
\end{equation}

Now let $k\in \mathcal{I}_{ij}$ and assume that for all $p\in\mathcal{I}_{ij}$ with $p<k$ we can write $E_{ij}^p$ as a linear combination of products of two elements of $\Phi(E(\mathcal{R})^d)$. We compute
\begin{align}
    \Phi(E_{ik})\Phi(E_{kj}) &= \sum_{C\in\mathcal{C}} E_{ik}^CE_{kj}^C \\
    &= E_{ik}^kE_{kj}^k+\sum_{p\in\mathcal{I}_{ij},p< k}E_{ik}^pE_{kj}^p\\
    &= E_{ij}^k + \sum_{p\in\mathcal{I}_{ij},p< k}E_{ij}^p.
\end{align}
Using the induction hypothesis we can now conclude that $E_{ij}^k$ can be written as a linear combination of products of two elements of $\Phi(E(\mathcal{R})^d)$.
\end{proof}
\begin{remark}
Note that for any $n\in \N$ there exists a chordal graph $G$ with $\diam(G)=n$. Hence the class of operator systems $E(\mathcal{R})$ with chordal graph $G(\mathcal{R})$ is a class of operator systems with arbitrary propagation numbers, however their duals always have propagation number at most two. In particular, this shows that if $\diam(G(\mathcal{R}))>2$, then $E(\mathcal{R})$ is not self-dual. By considering $C^*$-envelopes one sees that $E(\mathcal{R})$ is not self-dual even if $\diam(G(\mathcal{R}))=2$.
\end{remark}

\section{Band matrices}
In this section we will illustrate our results by applying them to operator systems of band matrices.

Let $X=\{1,\dots,n\}$, let $b\in\N$ be a bandwidth and let the tolerance relation $\mathcal{R}$ on $X$ be defined by $\mathcal{R}=\{(i,j)\in X\mid |i-j|<b\}$. The operator system $E_{n,b}=E(\mathcal{R})\subseteq M_n(\C)$ consists of band matrices with bandwidth $b$. Observe that the usual ordering of the integers constitutes a perfect elimination ordering of the graph $G(\mathcal{R})$ and hence this graph is chordal. This means that our results are applicable to the operator system $E_{n,b}$.

\begin{description}
        \item[A concrete realization:] Let us first use our main result to find a concrete realization of $E_{n,b}^d$. It is not hard to see that the set $\mathcal{C}$ of maximal cliques of $G(\mathcal{R})$ is given by \sloppy{${\mathcal{C}=\{C[i]\subseteq X\mid 1\leq i\leq n-b+1 \}}$}. Theorem \ref{thm:main} now implies that the map
    \begin{equation}
        \Phi: E_{n,b}^d\to \bigoplus_{k=1}^{n-b+1} M_b(\C):(x_{ij})\mapsto ((x_{ij})_{k\leq i,j\leq k+b-1})_{1\leq k\leq n-b+1}
    \end{equation}
    is a complete order embedding.
    \begin{figure}[b]
        \centering
        \[
        \begin{pmatrix}
        * & * & & \\ \cline{2-3}
        * & \multicolumn{1}{|c}{*} & \multicolumn{1}{c|}{*} & \\
         & \multicolumn{1}{|c}{*} & \multicolumn{1}{c|}{*} & * \\ \cline{2-3}
         & & * & * & *\\
         & & & * & *
        \end{pmatrix}
        \]
        \caption{The map $\Phi$ essentially picks out all the $b\times b$ blocks along the diagonal of a matrix in $E_{n,b}$}
    \end{figure}
    
    \item[The $C^*$-envelope:] By Theorem \ref{thm:C*env} the map $\Phi$ in fact realizes the $C^*$-envelope \sloppy{${\Cenv(E_{n,b}^d)=\bigoplus_{k=1}^{n-b+1} M_b(\C)}$} of $E_{n,b}^d$.
    
    \item[The propagation number:] If $b=1$ or $b=n$ the relation $\mathcal{R}$ is an equivalence relation and Theorem \ref{thm:prop} gives us that the map $\Phi$ is in fact surjective and $\prop(E_{n,b}^d)=1$. Otherwise we have $\prop(E_{n,b}^d)=2$. 
    
    \item[Extremal rays:] Finally, we can use Theorem \ref{thm:extremal} to describe the extremal rays of $(E^d_{n,b})_+$ and hence the pure states of $E_{n,b}$. Let $0\neq (x_{ij})\in (E_{n,b}^d)_+$. Since two vertices $i,j\in X$ are joined by an edge in $G(\mathcal{R})$ if and only if they are close enough together in the sense that $|i-j|<b$, one sees that the graph $H$ from Theorem \ref{thm:extremal} is connected if and only if there are not $b$ or more consecutive zeroes on the diagonal of $(x_{ij})$ in between any two nonzero diagonal entries. If this graph $H$ is connected then $(x_{ij})$ generates an extremal ray if and only if all components of $\Phi((x_{ij}))$ are of rank one or zero. In this case, as observed in Remark \ref{rmk:projection}, the matrix $(x_{ij})$ is the projection of some positive rank one matrix $(y_{ij})\in M_n(\C)_+$. The matrix $(y_{ij})$ must be a rank one projection and hence it can be written $(y_{ij})=vv^*$ for some vector $v\in\C^n$. In order for the matrix $H$ to be connected, we see that this vector $v$ cannot have a sequence of $b$ or more consecutive zero entries in between two nonzero entries of $v$. In this sense we see that extremal rays of $(E^d_{n,b})_+$ correspond to vectors $v\in \C^n$ that do not contain too many consecutive zeroes.
\end{description}

\bibliographystyle{plain}
\bibliography{references.bib}

\end{document}